\documentclass{amsart}[12pt]

\newtheorem{theorem}{Theorem}

\newtheorem{corollary}[theorem]{Corollary}

\theoremstyle{definition}

\newtheorem{remark}[theorem]{Remark}

\usepackage{amscd,amssymb}

\begin{document}

\title[Infinitely based varieties of Lie algebras]
{On varieties of Lie algebras\\
with infinite basis property}
\author[Vesselin S. Drensky]{Vesselin S. Drensky}
\thanks{Published as Vesselin S. Drensky, Infinitely based varieties of Lie algebras (Russian),
Serdica 9 (1983), No. 1, 79-82. Zbl 0525.17006, MR0725813.\\
Available online at http://www.math.bas.bg/serdica/1983/1983-079-082.pdf}
\date{}
\maketitle

\begin{abstract}
Over an arbitrary field of positive characteristic we construct an example of a locally finite variety of Lie algebras
which does not have a finite basis of its polynomial identities. As a consequence we construct varieties of Lie algebras with prescribed properties.
\end{abstract}

We shall use the notation in \cite{VL} and \cite{D1}.
In particular, all products are left normed: $xyz=(xy)z$ and $xy^n=x\underbrace{y\cdots y}_n$.
In \cite{VL} and \cite{D1} M.\,R. Vaughan-Lee and the author constructed examples of infinitely based varieties of Lie algebras over a field of positive characteristic.
A similar example constructed also Yu.\,G. Kleiman (unpublished). In this regard, Yu.\,A. Bahturin raised the question
about the existence of locally finite infinitely based varieties of algebras.
Quite recently I.\,B. Volichenko \cite{Vo} constructed such an example in the case of a field of characteristic 2.
Using the ideas from \cite{VL} and \cite{D1} we give the answer in the case of an arbitrary field of positive characteristic.

\begin{theorem}\label{Theorem 1}
Let $K$ be a field of positive characteristic $p$. Then the variety of Lie algebras over $K$ defined by the identities
\begin{equation}\label{eq 1}
(x_1x_2)(x_3x_4)\cdots(x_{2p-1}x_{2p})x_{2p+1}=0,
\end{equation}
\begin{equation}\label{eq 2}
((x_1x_2)(x_3x_4))((x_5x_6)(x_7x_8))=0,
\end{equation}
\begin{equation}\label{eq 3}
x_1x_2^{p^2+2}=0,
\end{equation}
\begin{equation}\label{eq 4}
(x_1x_2\cdots x_{k+2})(x_1x_2)^{p-1}=0,\quad k=1,2,\ldots,
\end{equation}
is locally finite and infinitely based.
\end{theorem}

We shall prove the theorem in several steps. We shall follow the exposition of \cite[\S 2]{D1}.

1. Let $T_n$ be the set of all subsets of the set $\{1,2,\ldots,2n-1\}$ and
let $A_n$ be an algebra with basis over $K$
\[
\{a_{\sigma s}^{(1)},a_{\tau s}^{(k)},b_s\mid\sigma,\tau\in T_n,\vert\tau\vert\text{ is odd, }k=2,3,\ldots,p-1,s=0,1,\ldots,p\}.
\]
For all $\tau\in T_n$, where $\vert\tau\vert$ is odd, we set $a_{\tau s}^{(p)}=b_s$ if $\vert\tau\vert=2n-1$ and $a_{\tau s}^{(p)}=0$ if $\vert\tau\vert<2n-1$.
We define the multiplication in $A_n$ in the following way:
\pagebreak
\[
a_{\sigma s}^{(k)}a_{\tau t}^{(l)}=(-1)^{\vert\sigma\vert}a_{\pi s+t}^{(k+l)},
\]
\[
\text{where }\pi=\sigma\cup\tau,\text{ if }k+l\leq p, s+t\leq p, \sigma\cap\tau=\varnothing, \vert\sigma\vert+\vert\tau\vert\equiv 1\text{ (mod }2),
\]
\[
a_{\sigma s}^{(k)}a_{\tau t}^{(l)}=0\text{ in all other cases.}
\]

Let us consider the set of linear transformations of $A_n$
\begin{equation}\label{eq 5}
\{g_{\lambda s},h_{\mu s}\mid\lambda,\mu\in T_n,\vert\lambda\vert \text{ is odd, }\vert\mu\vert\text{ is even },s=0,1,\ldots,p\},
\end{equation}
defined by the following equalities:

$a_{\sigma s}^{(k)}g_{\lambda t}=a_{\pi s+t}^{(k)}$, where $\pi=\sigma\cup\lambda$, if $s+t\leq p$, $\sigma\cap\lambda=\varnothing$, $\vert\sigma\vert$ is odd,

$a_{\sigma s}^{(k)}h_{\mu t}=ka_{\pi s+t}^{(k)}$, where $\pi=\sigma\cup\mu$, if $s+t\leq p$, $\sigma\cap\lambda=\varnothing$,

$a_{\sigma s}^{(k)}g_{\lambda t}=a_{\sigma s}^{(k)}h_{\mu t}=0$ in all other cases.

As in \cite{D1} we can prove that the algebra $A_n$ is a Lie algebra in the variety ${\mathfrak A}^2\cap{\mathfrak N}_p$
and the transformations (\ref{eq 5}) are derivations of the algebra $A_n$ which pairwise commute.
Let us denote by $B_n$ the split extension of the algebra $A_n$ by the abelian algebra of derivations $D_n$ generated by the transformations (\ref{eq 5}).
Following \cite{D1} it is easy to show that $B_n$ satisfies the identities (\ref{eq 1}) and (\ref{eq 2}).

2. Let us denote by $C_n$ the subalgebra of $B_n$ generated by the set
\[
\{a_{\sigma 0}^{(1)},g_{\lambda 0},h_{\sigma 1},h_{\mu 0}\mid\sigma=\varnothing,\lambda=\{1\},\mu=\{2,3\},\{4,5\},\ldots,\{2n-2,2n-1\}\}.
\]
The algebra $C_n$ has a basis as a $K$-vector space
\begin{equation}\label{eq 6}
\begin{aligned}
\{a_{\sigma s}^{(k)},g_{\lambda 0},h_{\mu 0},h_{\nu 1}\mid \sigma\in T_n\text{ with }a_{\sigma s}^{(k)}\in C_n,
k=1,2,\ldots,p,s=0,1,\ldots,p,\\
\lambda=\{1\},\mu=\{2,3\},\{4,5\},\ldots,\{2n-2,2n-1\},\nu=\varnothing\}.\quad\quad\quad\quad
\end{aligned}
\end{equation}
We shall prove that $C_n$ is nilpotent of class $2p+n$ and satisfies the identity (\ref{eq 3}):

In any nonzero product $l_1l_2\cdots l_m$ of elements in (\ref{eq 6}) the elements $a_{\sigma s}^{(k)}$ appear not more than $p$ times, the element $h_{\nu 1}$ appears not more than $p$ times
and every element $g_{\lambda 0}, h_{\mu 0}$ appears not more than once. Hence $m\leq 2p+n$ and $C_n\in{\mathfrak N}_{2p+n}$.

Let us assume that $m\geq p^2+2$. We shall prove that $x_1x_2^m=0$ in $C_n$. Let $l_1,l_2\in C_n$, $l_2=a+d$, where $a\in A_n$, $d\in D_n$. Then $l_1l_2=a_1\in A_n$.
We shall unwrap parentheses in the product $l_1l_2^m=(l_1l_2)l_2^{m-1}=a_1(a+d)^{m-1}$:
\begin{equation}\label{eq 7}
l_1l_2^m=\sum a_1c_1c_2\cdots c_{m-1},
\end{equation}
where $c_i=a$ or $c_i=d$, $i=1,2,\ldots,m-1$. If $a_1c_1c_2\cdots c_{m-1}\not=0$, then $a$ appears among $c_1,c_2,\ldots,c_{m-1}$ not more than $p-1$ times;
if $a$ appears exactly $p-1$ times, then $c_{m-1}=a$ (because $a_{\sigma t}^{(p)}$ belongs to the center of $C_n$).
Therefore the nonzero products in the sum (\ref{eq 7}) are of the form
\begin{equation}\label{eq 8}
a_1d^{k_1}ad^{k_2}a\cdots ad^{k_s}a^{\varepsilon},\quad s\leq p-1,\varepsilon=0,1,
\end{equation}
and $k_1+k_2+\cdots+k_s+(s-1)+\varepsilon=m-1$. Hence
\begin{equation}\label{eq 9}
m=k_1+k_2+\cdots+k_s+s+\varepsilon\leq k_1+k_2+\cdots+k_s+p.
\end{equation}
The element $d$ is a linear combination of derivations in (\ref{eq 6}), i.e. $d=\sum\alpha_id_i$, $\alpha_i\in K$, and the $d_i$'s commute pairwise. Then ($\text{char}K=p$)
\begin{equation}\label{eq 10}
\left(\text{ad}\sum\alpha_id_i\right)^p=\sum\alpha_i^p(\text{ad}d_i)^p,
\end{equation}
\begin{equation}\label{eq 11}
(\text{ad}d_i)^p=\begin{cases}
0,&\text{if }d_i=g_{\lambda 0},h_{\mu 0},\mu\not=\varnothing,\\
\text{ad}h_{\nu p},&\text{if }d_i=h_{\nu 1},\nu=\varnothing,
\end{cases}
\end{equation}
because $a_{\sigma s}^{(k)}h_{\nu 1}^p=k^pa_{\sigma s+p}^{(k)}=ka_{\sigma s+p}^{(k)}=a_{\sigma s}^{(k)}h_{\nu p}$, $\nu=\varnothing$.
Let us consider the product $a_1d^{k_1}\cdots ad^{k_s}a^{\varepsilon}$ in (\ref{eq 8}). Let
\begin{equation}\label{eq 12}
k_i=\max(k_1,\ldots,k_s),\quad k_j=\max(\{k_1,\ldots,k_s\}\setminus\{k_i\}).
\end{equation}
If $k_i\leq 2p-1$, $k_j\leq p-1$, then ($s\leq p-1$ from (\ref{eq 12}))
\[
k_1+\cdots+k_s\leq (2p-1)+(s-1)(p-1)\leq (2p-1)+(p-2)(p-1)=p^2-p+1
\]
and by (\ref{eq 9}) $m\leq p^2+1$ which is impossible. Hence there are two cases: $k_j\geq p$ or $k_i\geq 2p$.
For convenience of the notation we assume that $i=1$, $j=2$ (with the same considerations in the general case).

(i) Let $k_2\geq p$. Then, in virtue of (\ref{eq 10}) and (\ref{eq 11}),
\[
a_1d^{k_1}ad^{k_2}=a_1d^pd^{k_1-p}ad^pd^{k_2-p}=a_1\left(\sum\alpha_id_i\right)^pd^{k_1-p}a\left(\sum\alpha_id_i\right)^pd^{k_2-p}
\]
\[
=a_1\left(\sum\alpha_i^p\text{ad}^pd_i\right)d^{k_1-p}a\left(\sum\alpha_i^p\text{ad}^pd_i\right)d^{k_2-p}=\alpha^{2p}a_1h_{\nu 1}^pd^{k_1-p}ah_{\nu 1}^pd^{k_2-p}
\]
\[
=\alpha^{2p}a_1h_{\nu p}d^{k_1-p}ah_{\nu p}d^{k_2-p}=0,\quad \nu=\varnothing.
\]
(Here $\alpha$ is the coefficient of $d_{i_0}=h_{\nu 1}$ in $\sum\alpha_id_i$.)
Hence in this case the product (\ref{eq 8}) is equal to zero.

(ii) $k_1\geq 2p$. Then it follows from (\ref{eq 10}) and (\ref{eq 11}) that
\[
a_1d^{k_1}=a_1d^pd^pd^{k_1-2p}=a_1\left(\sum\alpha_id_i\right)^p\left(\sum\alpha_id_i\right)^pd^{k_1-2p}
\]
\[
=a_1\left(\sum\alpha_i^p\text{ad}^pd_i\right)^2d^{k_1-2p}=\alpha^{2p}a_1\text{ad}^{2p}h_{\nu 1}d^{k_1-2p}=\alpha^2a_1h_{\nu p}^2d^{k_1-2p}=0,\quad \nu=\varnothing.
\]
Again, the product (\ref{eq 8}) is equal to zero. Therefore, the algebra $C_n$ satisfies the identity (\ref{eq 3}).

3. The variety defined by the identities (\ref{eq 1}), (\ref{eq 2}), and (\ref{eq 3}) is solvable and satisfies the Engel identity. By the theorem of K.\,W. Gruenberg \cite{G}
it is locally nilpotent and hence locally finite.

4. As in \cite[Lemmas 2.1 and 2.2 and Theorem 2.3]{D1} we can prove that $C_n$ satisfies all identities (\ref{eq 4}) for $k<n$ and
$(x_1x_2\cdots x_{n+2})(x_1x_2)^{p-1}\not=0$ in $C_n$. Since $C_n$ is nilpotent of class $2p+n$ it satisfies also (\ref{eq 4}) for $k>n$, i.e., the identities (\ref{eq 4})
are independent modulo the identities (\ref{eq 1}), (\ref{eq 2}), and (\ref{eq 3}). This completes the proof of the theorem.

\medskip

Using ideas of A.\,Yu. Ol'shanski\v{\i} \cite{O1} and \cite{O2} and Theorem \ref{Theorem 1} we can construct some examples of varieties of Lie algebras with prescribed properties.

\begin{theorem}\label{Theorem 2}
Over any field of positive characteristic there exists a continuum of locally finite varieties of Lie algebras which form a strictly increasing chain with respect to the inclusion.
\end{theorem}

\begin{proof}
We order the set $\mathbb Q$ of the rational numbers:
\begin{equation}\label{eq 13}
r_1,r_2,r_3,\ldots,
\end{equation}
and consider the set of Lie algebras $\{L_r\mid r\in {\mathbb Q}\}$. Here $L_r\cong C_m$,
where $m$ is the number of $r$ in (\ref{eq 13}) and $C_m$ is the algebra from the proof of Theorem \ref{Theorem 1}.
For any real $\alpha$ we define the variety
\begin{equation}\label{eq 14}
{\mathfrak W}_{\alpha}=\text{var}\{L_r\mid r\leq\alpha\}.
\end{equation}
It follows from Part 4 of the proof of Theorem \ref{Theorem 1} that if $\alpha<\beta$, then ${\mathfrak W}_{\alpha}\subsetneqq{\mathfrak W}_{\beta}$.
\end{proof}

\begin{corollary}\label{Corollary 3}
Over any field of positive characteristic there exists a continuum of locally finite varieties of Lie algebras with pairwise different ordinal functions.
(The ordinal function $f_{\mathfrak V}:{\mathbb N}\to{\mathbb N}$ of the locally finite variety $\mathfrak V$ is defined by the equality
$f_{\mathfrak V}(n)=\dim F_n({\mathfrak V})$, $n=1,2,\ldots$, where $F_n({\mathfrak V})$ is the relatively free algebra of rank $n$ in $\mathfrak V$.)
\end{corollary}

\begin{proof}
Let us consider the varieties (\ref{eq 14}). If $\alpha>\beta$ then
\begin{equation}\label{eq 15}
{\mathfrak W}_{\beta}\subset{\mathfrak W}_{\alpha}
\end{equation}
and for every $n$ there is a canonical epimorphism $\psi_n:F_n({\mathfrak W}_{\alpha})\to F_n({\mathfrak W}_{\beta})$.
Since the inclusion (\ref{eq 15}) is strict, there is an $n$ such that $\ker\psi_n\not=0$, which implies that $f_{{\mathfrak W}_{\alpha}}(n)\not=f_{{\mathfrak W}_{\beta}}(n)$.
\end{proof}

\begin{remark}\label{Remark 4}
As in \cite{O1} we can prove that over a field of characteristic $p>0$ there exist two different varieties with the same ordinal functions.
For example, let $k=p^3+p^2+p+2$, ${\mathfrak M}={\mathfrak A}^2\cap{\mathfrak N}_k$. We consider the subvarieties $\mathfrak V$ and $\mathfrak W$ of $\mathfrak M$ defined, respectively, by the identities
\[
v(x_1,x_2,y_0,y_1,\ldots,y_p,y_{p+1})=x_1x_2y_0^py_1^p\cdots y_p^py_{p+1}^{p^3}=0,
\]
\[
w(x_1,x_2,y_0,y_1,\ldots,y_p,y_{p+1})=x_1x_2y_0^{p^2}y_1^{p^2}\cdots y_p^{p^2}y_{p+1}^p=0.
\]
Then $f_{\mathfrak V}(n)=f_{\mathfrak W}(n)$.
\end{remark}

\begin{remark}
Using the ideas from his paper \cite{M} Yu.\,N. Mal'tsev (unpublished) constructed an example of two just non-commutative varieties of associative algebras
over a finite field with the same ordinal functions. Using the ideas from \cite{O1} we can construct the following simple example.
Let $K$ be a field of characteristic $p>0$ and let $k=p^2+p+1$. Define the variety $\mathfrak M$ of associative algebras
$xy-yx=x_1x_2\cdots x_{k+1}=0$ and its subvarieties $\mathfrak V$ and $\mathfrak W$ defined, respectively by the identities
\[
v(x_0,x_1,\ldots,x_p,y)=x_0x_1\cdots x_py^{p^2}=0,\quad w(x_0,x_1,\ldots,x_p,y)=(x_0x_1\cdots x_p)^py=0.
\]
Then ${\mathfrak V}\not={\mathfrak W}$ but $\mathfrak V$ and $\mathfrak W$ have the same ordinal functions.
\end{remark}

The results of this paper are included in the Ph.\,D. thesis of the author \cite{D2}.
The author expresses his sincere gratitude to Yu.\,A. Bahturin for the stating of the problem and attention to the work.

\end{document}